\documentclass[a4paper]{article}
\usepackage{amsmath,amssymb,graphicx}

\newtheorem{prethm}{{\bf Theorem}}

\newenvironment{thm}{\begin{prethm}{\hspace{-0.5
               em}{\bf.}}}{\end{prethm}}

\newtheorem{prepro}[prethm]{Proposition}
\newenvironment{pro}{\begin{prepro}{\hspace{-0.5
               em}{\bf.}}}{\end{prepro}}

\newtheorem{prelem}[prethm]{Lemma}
\newenvironment{lem}{\begin{prelem}{\hspace{-0.5
               em}{\bf.}}}{\end{prelem}}

\newtheorem{precor}[prethm]{Corollary}

\newtheorem{prerem}[prethm]{{\bf Remark}}
\newenvironment{rem}{\begin{prerem}\em{\hspace{-0.5
              em}{\bf.}}}{\end{prerem}}

\newtheorem{preexample}{{\bf Example}}

\newtheorem{preproof}{{\bf Proof.}}

\newenvironment{proof}[1]{\begin{preproof}{\rm
               #1}\hfill{$\Box$}}{\end{preproof}}

\newcommand{\noi}{\noindent}
\newcommand{\te}{\theta}
\newcommand{\x}{{\bf x}}
\newcommand{\w}{{\bf w}}
\newcommand{\g}{\geqslant}
\newcommand{\lee}{\leqslant}
\newcommand{\p}{{\bf p}}
\newcommand{\e}{{\bf e}}
\newcommand{\B}{{\cal B}}
\newcommand{\C}{{\cal C}}
\newcommand{\D}{{\cal D}}
\newcommand{\la}{\lambda}
\newcommand{\al}{\alpha}
\newcommand{\be}{\beta}
\newcommand{\vf}{\varphi}
\newcommand{\tr}{{\rm tr}}

\DeclareMathAlphabet{\mathpzc}{OT1}{pzc}{m}{it}

\title{On optimality of designs with three distinct eigenvalues}

\author{M.R. Faghihi$^1$\quad E. Ghorbani$^{2,3,}$\thanks{Corresponding author, email: {\tt e\_ghorbani@ipm.ir}}
\quad G.B. Khosrovshahi$^3$\quad  S. Tat$^1$\\[.3cm]
{\small $^1$Department of Statistics, Shahid Beheshti University, Tehran, Iran} \\
{\small $^2$Department of Mathematics, K.N. Toosi University of Technology,  }\\
{\small P.O. Box 16315-1618, Tehran, Iran}  \\
{\small $^3$School of Mathematics, Institute for Research in Fundamental Sciences (IPM),}\\
{\small P.O. Box 19395-5746, Tehran, Iran }}

\begin{document}
\maketitle
\begin{abstract}

Let $\D_{v,b,k}$
denote the family of all connected block designs
with $v$ treatments and $b$ blocks of size $k$. Let $d\in\D_{v,b,k}$.
The  replication of a treatment is the number of times it appears in the blocks of $d$.
 The matrix $C(d)=R(d)-\frac{1}{k}N(d)N(d)^\top$ is called the information matrix of $d$ where $N(d)$ is the incidence matrix of $d$ and $R(d)$ is a diagonal matrix of the replications.
Since $d$ is connected, $C(d)$ has $v-1$ nonzero eigenvalues $\mu_1(d),\ldots,\mu_{v-1}(d)$.
Let $\D$ be the class of all binary designs of $\D_{v,b,k}$.
We prove that if  there is a design $d^*\in\D$ such that (i) $C(d^*)$ has three distinct eigenvalues, (ii) $d^*$ minimizes trace of $C(d)^2$ over $d\in\D$,  (iii) $d^*$ maximizes the smallest nonzero eigenvalue and the product of the nonzero eigenvalues of $C(d)$ over $d\in\D$,
then for all $p>0$, $d^*$ minimizes $\left(\sum_{i=1}^{v-1}\mu_i(d)^{-p}\right)^{1/p}$ over $d\in\D$.
In the context of optimal design theory, this means that if there is a design $d^*\in\D$ such that its information matrix has three distinct eigenvalues satisfying the condition (ii) above and that $d^*$ is E- and D-optimal in $\D$,
then $d^*$ is  $\Phi_p$-optimal in $\D$ for all $p>0$. As an application, we demonstrate the $\Phi_p$-optimality of certain group divisible designs.
 Our proof is based on the method of KKT conditions in nonlinear programming.

\vspace{3mm}
\end{abstract}

\section{Introduction}

For a statistician, a design is a much more general structure than what it means to a combinatorialist.
What statisticians consider a design is in fact a more general structure than a hypergraph; the blocks may contain repeated treatments. Statisticians use designs as experimenting schemes
 and they need to  decide which one is better than the other in some sense.
Their criterion for this is the general principle that  a better design has a smaller variance of estimators \cite{bc}. A design which is the best in this sense is called `optimal'.
Usually, statisticians consider three major criteria for optimality: A-optimality, D-optimality, and E-optimality. (The letters A, D, and E  stand for `average',  `determinant', and `extreme'.)
These criteria can be stated in terms of the eigenvalues of certain matrices associated to designs which is discussed in the remaining parts of this introductory remarks.

 We consider designs in statistical sense which are pairs $(X,\B)$ where $X$ is a $v$-set whose elements are called {\em treatments} and $\B$ is a collection of lists (called {\em blocks}) each consists of $k$ treatments.
A design is said to be {\em connected} if for every pair of treatments
it is possible to pass from one to the other through a chain of treatments such that any two consecutive treatments in the chain appear in a common block.
The set of all connected designs with $v$ treatments, $b$ blocks and block size $k$ is denoted by $\D_{v,b,k}$.
Since the blocks of a design are defined as lists, they may contain repeated elements. If the blocks are subsets of $X$, i.e. have no repeated elements, then the design is called {\em binary}.
Let $d\in\D_{v,b,k}$. Let $N(d)$ be the $v \times b$ {\em incidence matrix} whose $(i,j)$ entry is the number of times that treatment $i$ occurs in block $j$.
Thus the
column sums of $N(d)$ are all equal to $k$, the block sizes while the sum $r_i$ of the $i$-th row is the number of times which treatment $i$
occurs overall which is the {\em replication} of $i$. If $d$ binary, then $N(d)$ is a $(0,1)$-matrix.
A design is called {\em equireplicate} if all the treatments have equal replications.
The {\em concurrence matrix} of $d$ is the $v\times v$ matrix $S(d)=N(d)N(d)^\top$.
The diagonal matrix whose diagonal entries are the replication numbers of treatments is denoted by $R(d)$.
The matrix $C(d)=R(d)-\frac{1}{k}S(d)$ is called the {\em information matrix} of $d$.
It is well known that $C(d)$ is a positive semidefinite matrix and since $d$ is connected, $C(d)$ has exactly one eigenvalue zero.
Let $\mu_1(d)\g\cdots\g\mu_{v-1}(d)$ be the nonzero eigenvalues of $C(d)$ which we assume throughout  that are ordered decreasingly.
The multiset of nonzero eigenvalues of $C(d)$ is called the {\em spectrum} of $C(d)$.
If $\mu_1>\cdots>\mu_s$ are distinct nonzero eigenvalues of $C(d)$ with multiplicities $t_1,\ldots,t_s$, respectively, we use the notation $\{\mu_1^{t_1},\ldots,\mu_s^{t_s}\}$ for the spectrum of $C(d)$.

Given a class of designs, a design is said to be {\em A-optimal} if it maximizes the harmonic mean of $\mu_1,\ldots,\mu_{v-1}$ in that class.
A design is {\em D-optimal} if it maximizes the geometric mean of $\mu_1,\ldots,\mu_{v-1}$. A design is said to be {\em E-optimal}
if it maximizes the minimum of the nonzero eigenvalues of $C(d)$.
The eigenvalue optimality criteria are generalized by Kiefer \cite{k} to a much more general criterion called  {\em $\Phi_p$-optimality}.
For any $p>0$, a design is $\Phi_p$-optimal if it minimizes
$$\left(\frac{\sum_{i=1}^{v-1}\mu_i^{-p}}{v-1}\right)^\frac{1}{p}.$$
A-optimality corresponds to $p = 1$; the limit as $p\to 0$ gives D-optimality; the limit
as $p\to\infty$ gives E-optimality.
Cheng \cite{ch78} further  generalized the notion of optimality as follows.  Let $a$ be a large enough positive number and $f$ be a real valued function defined on the interval $(0,a)$.
Suppose that $f$ satisfies the conditions (i) $\lim_{x\to0^+}f(x)=\infty$, (ii) $f'<0$, (iii) $f''>0$, and (iv) $f'''<0$.
Then a design $d^*$ is called {\em type 1 optimal} if for any $d$ with the same number of treatments and blocks as  $d^*$ and for all functions $f$ satisfying the above properties, we have
$$\sum_{i=1}^{v-1}f(\mu_i(d^*))\lee\sum_{i=1}^{v-1}f(\mu_i(d)).$$
The specific functions $f(x)=x^{-p}$ and $f(x)=-\ln x$ give $\Phi_p$-optimality and $D$-optimality, respectively.

The notion of universal optimality introduced by Kiefer \cite{k} helps
in unifying the various optimality criteria. Let $M_v$ be the set of all $v \times v$
symmetric matrices with zero row  and column sums. Consider a
function $\Phi:M_v\to\mathbb{R}$  such that
\begin{itemize}
  \item[\rm(i)] $\Phi$ is convex,
  \item[\rm(ii)] $\Phi(bC)$ is a nonincreasing function of $b\g0$ for any $C\in M_v$, and
  \item[\rm(iii)] $\Phi$ is invariant under each simultaneous permutation of rows and
columns.
\end{itemize}
A design $d^*$ is said to be {\em universally
optimal} over a class of competing designs $\D$  if $d^*\in\D$ and for every function $\Phi$ satisfying the above conditions $\Phi(C(d^*))\lee\Phi(C(d))$ for any $d\in\D$. It can be shown that a design that is universally optimal is also
A-, D- and E-optimal.

The theory of optimal designs is discussed in details in the recent survey \cite{bc}.

In this paper we are interested in the optimality of designs with three distinct eigenvalues, that is designs $d$ for which the information matrix $C(d)$ has three distinct eigenvalues. For equireplicate designs, this is equivalent to say that the concurrence matrix of $d$ has three distinct eigenvalues. The (connected) designs with three distinct eigenvalues are
called {\em connected designs with second-order balance} in the statistical literature.
R.A. Bailey (see \cite{cam}) raised the question that which designs have three eigenvalues.
More specific, it was asked for which  equireplicate designs $d$ does the concurrence matrix $S(d)$ have three distinct
eigenvalues. This was partially answered in \cite{ds,ds2}.
This class of designs include {\em partial geometric designs}. A partial geometric design
is defined as a binary equireplicate connected design whose concurrence
matrix is a singular matrix with at most three distinct eigenvalues, see \cite{bbs,bss}.
 The optimality of designs with few eigenvalues has captured the attention of many workers in the field.

The following result due to Kiefer \cite{k} provides a sufficient condition
for determining a universally optimal design over a class of competing designs $\D$.
\begin{thm} {\rm(Kiefer \cite{k})} Suppose a class $\C = \{ C(d) \mid d\in\D\}$ of matrices in $M_v$
contains a $C(d^*)$ for which
\begin{itemize}
  \item[\rm(i)] $C(d^*)$ is completely symmetric, that is its diagonal elements are constant and
its off-diagonal elements are constant, and
\item[\rm(ii)] $d^*$ maximizes the trace of $C(d)$ over $d\in\D$.
  \end{itemize}
Then $d^*$ is universally optimal over $\D$.
In particular, if $\D_{v, b, k}$ contains a BIBD $d^*$, then $d^*$ is universally optimal over $\D_{v, b, k}$.
\end{thm}
In other words, the Kiefer's result says that if a design $d^*$ maximizes the trace of $C(d)$ over $d\in\D_{v, b, k}$ and the spectrum of $C(d^*)$ is of the form $\{\mu_1^{v-1}\}$, then
$d^*$ is universally optimal over $\D_{v, b, k}$.

\begin{thm} {\rm(Cheng \cite{ch78})} Suppose $d^*\in\D_{v,b,k}$ satisfies the following properties:
\begin{itemize}
\item[\rm(i)] $C(d^*)$ has spectrum of the form $\{\mu_1,\mu_2^{v-2}\}$,
\item[\rm(ii)] $d^*$ maximizes trace of $C(d)$ over $d\in\D_{v,b,k}$,
\item[\rm(iii)] $d^*$ minimizes the trace of $C(d)^2$ over $d\in\D_{v,b,k}$.
\end{itemize}
Then $d^*$ is type 1 optimal in $\D_{v,b,k}$.
\end{thm}

\begin{thm}\label{ch} {\rm(Cheng \cite{ch87})} If there is a design $d^*\in\D_{v, b, k}$ such that
\begin{itemize}
\item[\rm(i)] $C(d^*)$ has spectrum of the form $\{\mu_1^{v-2},\mu_2\}$,
\item[\rm(ii)] $d^*$ maximizes trace of $C(d)$ over $d\in\D_{v,b,k}$,
\item[\rm(iii)] $d^*$ is $\Phi_p$-optimal for some $p > 0$,
\end{itemize}
then $d^*$ is $\Phi_q$-optimal for all $0\lee q\lee p$.
\end{thm}
In the same paper, Cheng also showed that the same result holds if one replace the condition (ii) in Theorem~\ref{ch} by ``(ii)$'$ $d^*$ is D-optimal.''

\begin{thm} {\rm(Jacroux \cite{j})} Suppose $d^*\in\D_{v,b,k}$ satisfies the following properties:
\begin{itemize}
\item[\rm(i)] $C(d^*)$ has spectrum of the form $\{\mu_1,\mu_2^{v-3},\mu_3\}$,
\item[\rm(ii)] $d^*$ minimizes the trace of $C(d)^2$ over $d\in\D_{v,b,k}$,
\item[\rm(iii)] $d^*$ is E-optimal.
\end{itemize}
Then $d^*$ is type 1 optimal in $\D_{v,b,k}$.
\end{thm}

A design $d^*$ is said to be {\em $M$-optimal} (or {\em Schur-optimal}) in $\D_{v,b,k}$ if for any $d\in\D_{v,b,k}$, the vector of eigenvalues of $C(d)$, ordered decreasingly, majorizes the vector of eigenvalues of $C(d^*)$, that is
$$\sum_{i=1}^t\mu_i(d^*)\lee\sum_{i=1}^t\mu_i(d),~~~\hbox{for all $t=1,\ldots,v-1$.}$$
It is known that if $d^*$ is $M$-optimal, then it is optimal with respect to many criteria including type~1 optimality. We note that for a partial geometric design $d$, the replication $r$ is an eigenvalue of $C(d)$ as $S(d)$ is singular.
The {\em dual} of a design $d$ with $b$ blocks and $v$ treatments
is the design $\overline{d}$ with $v$ blocks, $b$ treatments and $N(\overline{d})=N(d)^\top$.

\begin{thm} {\rm(Bagchi and Bagchi \cite{bb})}
Suppose $d^*\in\D_{v,b,k}$ is a partial geometric design with replication $r$ and spectrum $\{r^g,\mu^{v-1-g}\}$.
If $d^*$ satisfies the following properties
\begin{itemize}
\item[\rm(i)] $g\lee\frac{(v-1)(k-1)}{r(v-k)}$,
\item[\rm(ii)]  the dual of $d^*$ is M-optimal in the class of all equireplicate designs of $\D_{b,v,r}$,
\end{itemize}
 then $d^*$ is M-optimal in $\D_{v,b,k}$.
\end{thm}
For similar results on M-optimality see \cite{bmm}. More results on optimality of designs with few eigenvalues can be found in \cite{b,ch78}.

In this paper, we continue this line of research and prove the following theorem.

\begin{thm}\label{main} Let $\D$ be the class of all binary designs of $\D_{v,b,k}$. Let $d^*\in\D$ such that $C(d^*)$ has two nonzero distinct eigenvalues.
If
 \begin{itemize}
   \item[\rm(i)] $d^*$ minimizes the trace of $C(d)^2$ over $d\in\D$,
    \item[\rm(ii)] $d^*$ is E-optimal in $\D$,
    \item[\rm(iii)] $d^*$ is D-optimal in $\D$,
   \end{itemize}
   then $d^*$ is $\Phi_p$-optimal for all $p>0$ in $\D$.
\end{thm}

\begin{rem} The referee pointed out that the designs which satisfy the hypotheses of Theorem~\ref{main} seem rather likely to be partially balanced designs with two associate classes and concurrences differing by one (for definition and properties see \cite[Chapter~11]{st}).
A related result was proved in \cite{cb} where it was shown that
among partially balanced designs with two associate classes and concurrences differing by one those which have a singular concurrence matrix are type 1 optimal within the subclass of all binary equireplicate incomplete designs of $\D_{v,b,k}$.
\end{rem}

As an application of Theorem~\ref{main}, we demonstrate the $\Phi_p$-optimality of certain group divisible designs.
 Group divisible designs are an important class of partially balanced incomplete
block designs. These designs
have $v$ treatments divided into $m$ groups of $n$ treatments each such that treatments
in the same group occur together in $\la_1$ blocks and treatments in different groups occur together in $\la_2$ blocks.
Jacroux \cite{j0} showed that group-divisible designs of group
size $2$, $k\g3$, and $\la_2=\la_1+1$ or $\la_2=\la_1-1$ (where $\la_1>1$) are D-optimal. These
designs have two distinct nonzero eigenvalues, and clearly minimize $\tr(C^2)$. They are
also E-optimal. The E-optimality of the former was shown by Takeuchi \cite{t} and the latter by Cheng \cite{ch80}. Hence by
Theorem~\ref{main}, they are $\Phi_p$-optimal. We summarize this in the following theorem.

\begin{thm}  Group-divisible designs of group
size $2$, $k\g3$, and $\la_2=\la_1+1$ or $\la_2=\la_1-1$ (where $\la_1>1$) are $\Phi_p$-optimal for all $p>0$.
\end{thm}
Another application of Theorem~\ref{main} concerning the optimality of the Petersen graph will be given at the end of Section~3.

Theorem~\ref{main} is a consequence of the following general inequality which could be of interest on its own.
We recall that for two different designs $d_1,d_2\in\D_{v,b,k}$, it is likely that ${\rm trace}\,C(d_1)\ne{\rm trace}\,C(d_2)$.
Nonetheless, for binary designs these are equal, namely for all binary designs $d\in\D_{v,b,k}$, ${\rm trace}\,C(d)=b(k-1)$.

\begin{thm}\label{main2} Let $(\te_1,\ldots,\te_n)$ be a vector consisting of two distinct, positive components.
  If a vector $(x_1,\ldots,x_n)$ of positive components satisfies the conditions
\begin{itemize}
   \item[\rm(i)] $x_1+\cdots+x_n=\te_1+\cdots+\te_n$,
    \item[\rm(ii)] $x_1^2+\cdots+x_n^2\g \te_1^2+\cdots+\te_n^2$,
    \item[\rm(iii)] $\min\{x_i \mid i=1,\ldots,n\}\lee\min\{\te_i \mid i=1,\ldots,n\}$,
    \item[\rm(iv)] $\prod_{i=1}^n x_i \lee \prod_{i=1}^n \te_i $,
   \end{itemize}
   then for all $p>0$,
$$x_1^{-p}+\cdots+x_n^{-p}\lee \te_1^{-p}+\cdots+\te_n^{-p}.$$
\end{thm}
The proof of Theroem~\ref{main2} is based on Karush--Kuhn--Tucker (KKT) conditions from nonlinear programming and shall be presented in Section~3.

\section{Karush--Kuhn--Tucker (KKT) conditions}

In nonlinear programming, the Karush--Kuhn--Tucker (KKT) conditions are
necessary for a local solution to a minimization problem provided that some regularity conditions are
satisfied. Allowing inequality constraints, the KKT approach to nonlinear programming generalizes the method of
Lagrange multipliers, which allows only equality constraints. For details see \cite{book}.

Consider the following optimization problem:
\begin{quote}
   Minimize $f(\x)$\\
   subject to:\\
   $~~~~~g_i(\x)\lee0$,~ for $i\in I$,\\
   $~~~~~h_j(\x)=0$,~ for $j\in J$,
\end{quote}
where $I$ and $J$ are finite sets of indices.
Suppose that the objective function $f:\mathbb{R}^n\to\mathbb{R}$ and the constraint functions $g_i:\mathbb{R}^n\to\mathbb{R}$ and
$h_j:\mathbb{R}^n\to\mathbb{R}$ are continuously differentiable at a point $\x^*$. If $\x^*$ is a local minimum that satisfies some regularity conditions, then there exist constants $\nu_i$ and $\la_j$, called
KKT multipliers, such that
\begin{align*}
\nabla f(\x^*)+\sum_{i\in I}&\,\nu_i\nabla g_i(\x^*)+\sum_{j\in J}\la_j\nabla h_j(\x^*)={\bf0}\\
  g_i(\x^*)&\lee0,~~~\hbox{for all $i\in I$},\\
   h_j(\x^*)&=0,~~~\hbox{for all $j\in J$},\\
    \nu_i&\g0,~~~\hbox{for all $i\in I$},\\
     \nu_ig_i(\x^*)&=0,~~~\hbox{for all $i\in I$}.
\end{align*}

In order for a minimum point to satisfy the above KKT conditions, it should satisfy some regularity conditions (or constraint qualifications).  The one which suits our problem is the Mangasarian--Fromovitz constraint qualification (MFCQ).
Let $I(\x^*)$ be the set of indices of active inequality constraints at $\x^*$, i.e.
$I(\x^*)=\left\{i\in I\mid  g_i(\x^*)=0\right\}$. We say that MFCQ holds at a feasible point $\x^*$
if the set of gradient vectors
$\{\nabla h_j(\x^*)\mid j\in J\}$ is linearly independent and that there exists $\w\in\mathbb{R}^n$ such that
\begin{align*}
  \nabla g_i(\x^*)\w^\top&<0,~~~  \hbox{for all $i\in I(\x^*)$},\\
  \nabla h_j(\x^*)\w^\top&=0,~~~ \hbox{for all $j\in J$}.
\end{align*}

\begin{thm} {\rm(\cite{mf}, see also \cite{book})} If a local minimum $\x^*$ of the function  $f(\x)$ subject to the constraints $g_i(\x)\lee0$, for $i\in I$, and   $h_j(\x)=0$, for $j\in J$, satisfies MFCQ, then it satisfies the KKT conditions.
\end{thm}

\section{Proofs}

In this section, we prove Theorem~\ref{main2}.
We start by stating some results on inequalities.

\begin{lem}\label{ben1} {\rm (Bennet \cite{ben})} Suppose that $\al_1,\al_2,\delta_1,\delta_2\g0$, $d_1<a_1<a_2<d_2$,
$\al_1+\al_2=\delta_1+\delta_2$, and that $\al_1 a_1+\al_2 a_2=\delta_1 d_1+\delta_2 d_2$.
 If $\vf$ is a convex function, then $$\al_1\vf(a_1)+\al_2\vf(a_2)\lee\delta_1\vf(d_1)+\delta_2\vf(d_2).$$
\end{lem}

\begin{lem}\label{ben2} {\rm (Bennet \cite{ben})} Suppose that $\al_1,\al_2,\delta_1,\delta_2\g0$, $a_1<d_1<a_2<d_2$,
$\al_1+\al_2=\delta_1+\delta_2$, $\al_1 a_1+\al_2 a_2=\delta_1d_1+\delta_2 d_2$, and that $\al_1 a_1^2+\al_2 a_2^2\g\delta_1 d_1^2+\delta_2 d_2^2$.
 If $\vf$ is a concave and $\vf'$  a convex function, then $$\al_1\vf(a_1)+\al_2\vf(a_2)\lee\delta_1\vf(d_1)+\delta_2\vf(d_2).$$
\end{lem}

\begin{lem}\label{2eq} Let $m,n,s,t$ be positive integers with $m+n=s+t$ and $a_1,a_2,x,y$ be reals with $0<x\lee a_1<a_2\lee y$. If $sx+ty=ma_1+na_2$ and $sx^2+ty^2=ma_1^2+na_2^2$,
then $m=s, n=t, x=a_1$, and $y=a_2$.
\end{lem}
\begin{proof}{First assume that $s\g m$. If we let
$\al_1:=x-a_1\lee0$, $\al_2:=x-a_2<0$, and $\gamma:=y-a_2\g0$,
then, by the assumption, $m\al_1+(s-m)\al_2+t\gamma=0.$
Now we have
\begin{align*}
    sx^2+ty^2&=m(a_1+\al_1)^2+(s-m)(a_2+\al_2)^2+t(a_2+\gamma)^2 \\
    &=ma_1^2+na_2^2+m\al^2+(s-m)\al_2^2+t\gamma^2+2m\al_1(a_1-a_2)+2a_2(m\al_1+(s-m)\al_2+t\gamma).
\end{align*}
By the assumption, it is necessary that $s=m$ and $\al_1=\gamma=0$ which implies the result. The case $s<m$ can be handled similarly.
}\end{proof}

Now we let   $$\p:=(\te_1,\ldots,\te_1,\te_2,\ldots,\te_2)\in\mathbb{R}^n,$$
such that $0<\te_1<\te_2$,  $\te_1$ is repeated $m_1$ times and $\te_2$ is repeated $m_2$ times.
In order to prove Theorem~\ref{main2}, we fix $p>0$ for the rest of the paper and
  find the global minima of the function  $$f(\x):=x_1^{-p}+\cdots+x_n^{-p},~~~ \x=(x_1,\ldots,x_n)\in\mathbb{R}^n$$ subject to the constraints:
\begin{align*}
    g(\x)&:=x_1+\cdots+x_n-m_1\te_1-m_2\te_2=0,\\
    h(\x)&:=m_1\te_1^2+m_2\te_2^2-x_1^2-\cdots-x_n^2\lee0,\\
    k(\x)&:=x_1\cdots x_n-\te_1^{m_1}\te_2^{m_2}\lee0,\\
    l_1(\x)&:=x_1-\te_1\lee0,\\
    l_2(\x)&:=\xi-x_1\lee0,\\
    m_i(\x)&:=x_1-x_i\lee0,~~\hbox{for $i=2,\ldots,n-1$},\\
    n_i(\x)&:=x_i-x_n\lee0,~~\hbox{for $i=2,\ldots,n-1$}.
\end{align*}
The positive number $\xi$ is chosen small enough so that it satisfies $\xi^{-p}>f(\p)$.

Hereafter we suppose that the vector $$\e=(e_1,\ldots,e_n)$$ is a local minimum of the above problem.

\begin{lem}\label{mfcq} If $l_2(\e)<0$, then $\e$ satisfies MFCQ.
\end{lem}
\begin{proof}{With no loss of generality, assume that $e_1\lee e_2\lee\cdots\lee e_n$. We have also $e_1<e_n$.
  The only equality constraint is $g(\x)=0$ for which $\nabla g(\x)$ is the all one vector. So if $\nabla g(\e)\w^\top=0$, then the components of $\w$ must sum up to zero. We also observe that $m_i(\e)=0$ and $n_i(\e)=0$ cannot simultaneously occur for any
$i=2,\ldots,n$.
Assume that  $t$ of $e_i$ are equal to $e_n$ for some $t\g1$ and $a,b,c$ be positive numbers with $a=(t-1)b+c$ and $c>b$.
By choosing $\w=(-a,0,\ldots,0,b,\ldots,b,c)$, with $b$ repeated  $t-1$ times, we see that the for all
$${\bf y}\in\{\nabla h(\e),\nabla k(\e),\nabla l_1(\e)\}\cup\{\nabla m_i(\e)\mid i=2,\ldots,n-1-t \}\cup\{\nabla n_i(\e)\mid i=n-2-t,\ldots,n-1 \}$$
 we have ${\bf y}\w^\top<0$. Hence MFCQ conditions are satisfied for $\e$.
}\end{proof}

\begin{thm}\label{ep} If $\e$ is a global minimum, then it must be a permutation of $\p$.
\end{thm}
\begin{proof}{Assume that $f(\e)\lee f(\p)$ and $e_1\lee e_2\lee\cdots\lee e_n$. We show that $\e$ must be equal to $\p$.
By the choice of $\xi$, $l_2(\e)<0$ and so by Lemma~\ref{mfcq}, $\e$ satisfies KKT conditions, namely
\begin{equation}\label{nabla}
   \nabla f(\e)+\nu\nabla g(\e)+\la\nabla h(\e)+\rho\nabla k(\e)+\sum_{i=1}^2\eta_i\nabla l_i(\e)+
\sum_{i=2}^{n-1}\left(\al_i\nabla m_i(\e)+\be_i\nabla n_i(\e)\right)=0,
\end{equation}
\begin{align}
&~~~~~~~~~~e_1+\cdots+e_n-m_1\te_1-m_2\te_2=0,\\
\la&\g0,~~~\la(m_1\te_1^2+m_2\te_2^2-e_1^2-\cdots-e_9^2)=0,\label{lambda}\\
\rho&\g0,~~~\rho(e_1\cdots e_n-\te_1^{m_1}\te_2^{m_2})=0,\label{prod}\\
\eta_1&\g0,~~~\eta_1(e_1-\te_1)=0,\label{eta1}\\
\eta_2&\g0,~~~\eta_2(\xi-e_1)=0,\label{eta3}\\
\al_i&\g0,~~~\al_i(e_1-e_i)=0,~~ \hbox{for $i=2,\ldots,n-1$},\label{alpha}\\
\be_i&\g0,~~~\be_i(e_i-e_n)=0,~~ \hbox{for $i=2,\ldots,n-1$}.\label{beta}
\end{align}
 Since $l_2(\e)<0$, we have $\eta_2=0$. If we let $D=\prod_{i=1}^ne_i$, then  (\ref{nabla}) can be written as
\begin{align}
  &-pe_1^{-p-1}+\nu-2\la e_1+\rho\frac{D}{e_1}+\eta_1+\al_2+\cdots+\al_{n-1}=0,\label{e1}\\
  &-pe_i^{-p-1}+\nu-2\la e_i+\rho\frac{D}{e_i}-\al_i+\be_i=0,~~ \hbox{for $i=2,\ldots,n-1$},\nonumber\\
   &-pe_n^{-p-1}+\nu-2\la e_n+\rho\frac{D}{e_n}-\be_2-\cdots-\be_{n-1}=0.\label{e9}
\end{align}
Assume that  $r$ of $e_i$ are equal to $e_1$, $t$ of $e_i$ are equal to $e_n$, and $s$ of them are between $e_1$ and $e_n$.
We consider four cases according to whether $r$ and $t$ are equal to $1$ or not.

\noi{\bf Case 1.} $r\g2$ and $t\g2$. We have $e_2=e_1$. This implies that $e_2<e_n$ and so $\be_2=0$ and
$-\al_2=\eta_1+\al_2+\cdots+\al_{n-1}$. Since $\eta_1,\al_i\g0$ it follows that $\eta_1=\al_2=\cdots=\al_{n-1}=0$.
Similarly $\be_2=\cdots=\be_{n-1}=0$.
It turns out that each $e_i$ must be a zero of the function
$$y(x):=-p+\nu x^{p+1}-2\la x^{p+2}+\rho D x^p.$$
It is easily seen that $y'(x)$ has at most one positive zero and thus $y(x)$ has at most two positive zeros.
Therefore,  each $e_i$ is  equal to either $e_1$ or $e_n$.
In case $\la=0$, $y(x)$ becomes a monotone function and thus it has at most one zero implying
 that $e_1=\cdots=e_n$ which is impossible. Therefore, $\la>0$ and so by (\ref{lambda}),
$$re_1^2+te_n^2=m_1\te_1^2+m_2\te_2^2.$$
Now Lemma~\ref{2eq} implies that $e_1=\te_1$, $e_n=\te_2$, $r=m_1$ and $t=m_2$. Therefore, $\e=\p$.

\noi{\bf Case 2.} $r\g2$ and $t=1$. Since $t=1$, by (\ref{beta}), all $\be_i$ are zero. Since $r\g2$, as above,
$\eta_1=\al_2=\cdots=\al_{n-1}=0$. It follows that all $e_i$ are zeros of $y(x)$ and so $e_1=\cdots=e_{n-1}$.
We have necessarily  $e_n\g\te_2$. Now, if $m_2>1$, then by Lemma~\ref{2eq}, we arrive at a contradiction and if $m_2=1$, then $\e=\p$.

\noi{\bf Case 3.} $r=1$ and $t\g2$. Since $r=1$, by (\ref{alpha}), all $\al_i$ are zero and from $t\g2$ it follows that $\be_2=\cdots=\be_{n-1}=0$.
 If $\eta_1=0$, then all $e_i$ admit at most two different values which is only possible when $r=m_1$ and $t=m_2$ as in Case~2.
Thus $\eta_1>0$ and so $e_1=\te_1$.
The rest of $e_i$ are zeros of $y(x)$ and so they are equal to either $e_2$ or $e_n$. Clearly $e_2>\te_1$. From lemma~\ref{2eq} it follows that  $e_n > \te_2$.
Now we apply Lemma~\ref{ben2} with the function $\vf(x)=\ln x$. It turns out that $e_1\cdots e_n>\te_1^{m_1}\te_2^{m_2}$, a contradiction.

\noi{\bf Case 4.} $r=t=1$. As above, we have  all $\al_i$ and $\be_i$ equal to zero and so all of $e_2,\ldots,e_{n}$ are zeros of $y(x)$.
From (\ref{e1})  it is clear that $y(e_1)\lee0$. So if we denote the zeros of $y(x)$ by $y_1$ and $y_2$ with $y_1\lee y_2$, then we have $e_1<y_1<e_n= y_2$ implying that $e_2=\cdots=e_{n-1}=y_1$.
If $\eta_1=0$, then $e_1$ must be a zero of $y(x)$ and it has to be equal to $y_1$ which is not possible.
So $\eta_1>0$ implying that $e_1=\te_1$. This yields the same result as in Case~3.

Consequently, we found that all the cases lead to a contradiction except for the case $r=m_1$ and $t=m_2$ which in turn implies that $\e=\p$.
The proof is now complete.
}\end{proof}

\noi{\bf Example.}
The celebrated Petersen graph has many fascinating properties. Concerning optimality, this distinguished graph shows another interesting and unique character.
In \cite{fgkt}, an algorithm is developed which searches for optimal designs within $\D_{v,b,k}$.
Implementing that algorithm, we looked for the optimal designs in the family of graphs with 10 vertices and 15 edges.
As result, the Petersen graph was pumped out as the A-, D-, and E-optimal design in that family. This was in fact one of our motivations for this work.
Now, as a  demonstration of Theorem~\ref{main}, we prove the general optimality property that
 {\em for all $p>0$, the Petersen graph is the unique $\Phi_p$-optimal and also D- and E-optimal graph among all  connected simple graphs on $10$ vertices and $15$ edges.}
This  follows from Theorem~\ref{main} and Proposition~\ref{te} below. Note that, since we restrict to simple graphs, that the Petersen graph minimizes $\tr(C^2)$ is trivial. Meanwhile, the uniqueness of Petersen graph as D- and E-optimal design is concluded from the equality  cases of Proposition~\ref{te}. (We remark that the E-optimality of Petersen graph is a special case of Theorem~3.3 of \cite{ch2012}. Nonetheless we include the short reasoning for the sake of completeness.) In passing we mention that in the case of graphs,
 the information matrix is
half its Laplacian matrix so one can consider the eigenvalues of Laplacian matrix for studying optimality. Recall that the Laplacian eigenvalues of Petersen graph are $\{5^4,2^5,0\}$. The Petersen graph is also uniquely determined by its Laplacian eigenvalues.

\begin{pro}\label{te} Let $G$ be a connected graph with $10$ vertices, $15$ edges, and denote the eigenvalues of Laplacian matrix of $G$ by $\mu_1\g\cdots\g\mu_9>\mu_{10}=0$. Then
 \begin{itemize}
   \item[\rm(i)] $\mu_9\lee2$ and the equality holds only for Petersen graph,
   \item[\rm(ii)] $\prod_{i=1}^9\mu_i\lee2\cdot10^4$  and the equality holds only for Petersen graph.
 \end{itemize}
\end{pro}
\begin{proof}{Let $d_1\g\cdots\g d_{10}$ be the degree sequence of $G$.
 It is known that for non-complete graphs, the smallest nonzero eigenvalue of Laplacian does not exceed the minimum degree (see, e.g., \cite[p.~198]{crs}).
 Hence $\mu_9\lee d_{10}$.
It follows that if $G$ is not regular, then $\mu_9\lee2$.
So we may assume that $G$ is regular. It is a well known fact that there are exactly 21 3-regular graphs on 10 vertices out of which 19 are connected.
 By inspecting the table of spectra of small graphs \cite{cds}, one can verify that the 19 3-regular graphs on 10 vertices satisfy (i).
There are exactly $112,618$ connected graphs on 10 vertices and 15 edges which can be extracted from the McKay's database on small graphs \cite{mc}.
By a simple computation one can verify (ii) and also the equality case of (i).
}\end{proof}

\section*{Acknowledgments}
We thank the referee for several fruitful comments.
The research of the second author was in part supported by a grant from IPM (No. 91050114).

\end{document}